\documentclass[12pt,reqno]{amsart}
\usepackage{enumerate,amsfonts,amsmath,amsthm,amssymb,latexsym,verbatim,amscd,mathrsfs,pb-diagram, graphics,color}
\usepackage{hyperref}
\usepackage{setspace}
\usepackage{geometry}
\allowdisplaybreaks
\theoremstyle{plain}
\newtheorem{thm}{Theorem}[subsection]

\newtheorem{lem}[thm]{Lemma}
\newtheorem{prop}[thm]{Proposition}

\theoremstyle{definition}

\newtheorem{rem}[thm]{Remark}

\numberwithin{equation}{subsection}

\begin{document}

\title[Integral Bases for twisted multiloop Algebras]{Integral Bases for twisted multiloop Algebras with diagram automorphism actions}\author{Angelo Bianchi} 
\address{Universidade Federal de S\~ao Paulo - UNIFESP - Department of Science and Technology, Brazil}
\email{acbianchi@unifesp.br}
\thanks{A.B. is partially supported by FAPESP grant 2024/14914-9.}

\author{Samuel Chamberlin}
\address{Department of Natural and Physical Sciences\\
Park University\\
Parkville, MO 64152}
\email{schamberlin@park.edu}

\begin{abstract}
We construct integral forms for the universal enveloping algebras of certain twisted multiloop algebras and explicit integral bases for these integral forms.
\end{abstract}

\maketitle

\section*{Introduction}

Let $\mathfrak g$ be a simple finite dimensional complex Lie algebra and $\mathbb C\langle n\rangle:=\mathbb C[t_1^{\pm 1},\dots, t_n^{\pm1}]$ be a Laurent polynomial ring in $n$ commutating variables. Then $\mathfrak g \otimes \mathbb C\langle n\rangle$ has a natural Lie algebra structure called \textit{multiloop Lie $n$-algebra} (associated to $\mathfrak g$) and its universal central extension is what is called a \textit{toroidal Lie $n$-algebra}. The multiloop and the toroidal algebras are both explored independently.

These toroidal Lie algebras are multi variable generalizations of the well known
affine Kac-Moody Lie algebras, which are universal central extensions of the loop algebra. Moody, Rao, and Yokonuma introduced the toroidal Lie algebras in \cite{RMY}, which led to constructions of representations of toroidal Lie algebras from vertex operators in works such as \cite{RM,R}. Analogously to the twisted affine Lie algebras, Fu and Jiang \cite{FJ} used diagram automorphisms with finite order of the base Lie algebra  to construct twisted toroidal Lie algebras. The structural aspects and the representation theory of toroidal Lie algebras are replete with interesting nuances when $n\ge 2$ compared to $n=1$, which require different techniques and approaches.   

If, on the one hand, the construction of toroidal algebras is quite natural as a generalization of affine Lie algebras, on the other hand, the twisted toroidal algebras are not so immediate. We basically follow the same approach as in \cite{Ba,BR,BK} using only one diagram automorphism as in the affine Kac-Moody algebras, instead of a set of them as in \cite{CJK,Le}.  The motivation to consider the action of the automorphism on the toroidal Lie algebra comes from the context of affine Kac-Moody algebras where the action on $\mathbb C[t^{\pm 1}]$ is well understood and the action on the base Lie algebra $\mathfrak g$ has the action by diagram automorphisms as the cornerstone as follows: if $\sigma$ is an arbitrary automorphism of $\mathfrak g$ of finite order $m$, then there exists an associated diagram automorphism $\mu$ (where the order of $\mu$ is $1, 2$, or $3$, depending on $\mathfrak g$), and an inner automorphism $\phi$ such that $\sigma = \mu \phi $ and the twisted Lie algebra associated to $\sigma$ is isomorphic to the one associated to $\mu$ (cf. \cite[Proposition 8.1 and Theorem 8.5]{Ka}).  

Another way to view the (twisted) toroidal Lie algebras is as a central extension of (equivariant) \textit{map algebras associated to $\mathfrak g$ and $X=Spec(\mathbb C\langle n \rangle)$} which is the Lie algebra of regular maps $\mathrm X \to \mathfrak g$ with pointwise Lie bracket. We highlight the reference \cite{NSS2} as a great survey on equivariant map algebras.

The theory of integral forms for the universal enveloping algebra of a Lie algebra has been developed since its important role in the development of the theory of Chevalley groups.  From 1955 to 1985, these forms were considered for finite dimensional complex simple Lie algebras, for their nontwisted and twisted analogues, and  for vertex algebras (cf. \cite{G,H,M,Pre}). After a 25-year hiatus, it has entered in the theory of Lie superalgebras (cf. \cite{BaC}), the Onsager algebra (cf. \cite{BiC1}), and even in the theory of map algebras (cf. \cite{C}), besides some computational simplifications still in the context of low rank twisted affine Kac-Moody algebras (cf. \cite{ DP}). 

Integral bases are bases whose $\mathbb Z$-span is the integral form. Their construction depends on straightening
identities in the universal enveloping algebra which are the main tool to write elements in Poincar´e-Birkhoff-Witt (PBW) order. One application of integral forms is that they allow one to study representation theory in positive characteristic through what is called a \textit{hyperalgebra}. This was first done in \cite{JM} and extended to several contexts in \cite{BiC2,BiC3,BMa,BMM} with both present authors figuring in.  

Putting these ideas together, we focus on the construction of integral forms for universal enveloping algebras of the twisted multiloop algebras constructed from the action of a diagram automorphism. Further, we explicit an integral basis for this integral form and straightening formulas for products of basis elements for this integral form.

Some natural consequences of this work include the study of the behavior of different automorphism actions and the representations of these twisted multiloop algebras over a field of positive characteristic.

\

\noindent \textbf{Organization of the paper:} Section \ref{pre} is dedicated to the algebraic preliminaries, including a review of diagram automorphisms, Chevalley basis, and the construction of the main object of this paper. In Section \ref{main}, we construct integral forms and an integral bases for these algebras. 

\

\tableofcontents

\section{Preliminaries} \label{pre}

Throughout this work, $\mathbb C$ denote the set of complex numbers and $\mathbb Z$ (respectively $\mathbb Z_+$ and $\mathbb N$) is the set of integers (respectively non-negative and positive integers).

\subsection{Diagram automorphisms}

Let $\mathfrak g$ be a simple complex finite-dimensional Lie algebra. Fix a Cartan subalgebra, $\mathfrak h$, a corresponding triangular decomposition $\mathfrak g=\mathfrak n^-\oplus\mathfrak h\oplus\mathfrak n^+$, and denote by $R$ the associated root system and by $R^+$ the set of positive roots. Let $I$ be the set of vertices of the Dynkin diagram $\mathfrak g$, and denote by $\{\alpha_i\mid i\in I\}$ the set of simple roots. 

Given $\alpha\in R$, let $\mathfrak g_\alpha$ denote the root space of $\alpha$. For convenience, if $\alpha\notin R$, define $x_\alpha^\pm:=0$.

Let $\mathcal C=\{x_\alpha^\pm,h_i\mid\alpha\in R^+,\ i\in I\}$ be a Chevalley basis for $\mathfrak g$ and define $x_i^\pm:=x_{\alpha_i}^\pm$ and $h_\alpha:=\left[x_\alpha^+,x_\alpha^-\right]$.

Giving a diagram automorphism $\sigma$ of $\mathfrak g$ it is well known that its order is $1,2$ or $3$, according to the type of the Lie algebra $\mathfrak g$. Further, the only types which admit such nontrivial automorphism are the one of types $A, D$, and $E$ (cf. \cite{Ka}). 

In this paper we fix a diagram automorphism $\sigma$ of $\mathfrak g$ and let $k\in\{1, 2,3\}$ be its order. Then $\sigma$ induces a permutation of $R$ as follows
$$\sigma\left(\sum_{i\in I}n_i\alpha_i\right)=\sum_{i\in I}n_i\alpha_{\sigma(i)}.$$
Also, $\sigma$ defines a unique Lie algebra automorphism of $\mathfrak g$, for all $i\in I$, given by $\sigma\left(x_i^\pm\right)=x_{\sigma(i)}^\pm$. Given $\alpha\in R^+$, define $k_\alpha:=\left|\left\{\sigma^j(\alpha)\mid j\in\mathbb Z\right\}\right|$. Then $k_\alpha\in\{1,k\}$. Notice also that, for all $\alpha\in R$, $\sigma\left(\mathfrak g_\alpha\right)=\mathfrak g_{\sigma(\alpha)}$. 

Let $\xi$ be a primitive $k$th root of unity and define $\mathfrak g_\epsilon:=\left\{x\in\mathfrak g\mid\sigma(x)=\xi^\epsilon x\right\}$. Then,
$$\mathfrak g=\bigoplus_{\epsilon=0}^{k-1}\mathfrak g_\epsilon\ \text{ and }\ \left[\mathfrak g_\epsilon,\mathfrak g_{\epsilon'}\right]\subset\mathfrak g_{\overline{\epsilon+\epsilon'}},$$
where $\overline{\epsilon+\epsilon'}$ is the remainder upon dividing $\epsilon+\epsilon'$ by $k$. From now on, by simplicity, we will abuse notation and write $\epsilon$ in place of $\overline{\epsilon}$. Notice that $\mathfrak g_0$ is a simple Lie subalgebra of $\mathfrak g$ and, for all $\epsilon\in\{0,\ldots,k-1\}$, $\mathfrak g_\epsilon$ is a simple $\mathfrak g_0$-module, \cite[Chapter VIII]{Ka}.

Let $\mathfrak a$ be any subalgebra of $\mathfrak g$, $\epsilon\in\{0,\ldots,k-1\}$, and denote $\mathfrak a_\epsilon=\mathfrak a\cap\mathfrak g_\epsilon$. $\mathfrak h_0$ is a Cartan subalgebra of $\mathfrak g_0$ and $\mathfrak g_0=\mathfrak n_0^+\oplus\mathfrak h_0\oplus\mathfrak n_0^-$ is a triangular decomposition of $\mathfrak g_0$. Let $I_0$ be the set of vertices of the Dynkin diagram of $\mathfrak g_0$ and $R_0$ the root system determined by $\mathfrak h_0$. We will denote the simple roots of $\mathfrak g_0$ as $\alpha_i$ for each $i\in I_0$. Let $R_0^+$ denote the set of positive roots and $R_s$ and $R_l$ the subsets of $R_0^+$ made up of the short and long roots of $\mathfrak g_0$, respectively. Define $Q_0$ to be the root lattice and $Q_0^+:=\bigoplus_{i\in I_0}\mathbb Z_+\alpha_i.$

Given $\epsilon\in\{0,\ldots,k-1\}$, let wt$\left(\mathfrak g_\epsilon\right)$ be the set of weights of $\mathfrak g_\epsilon$ as a $\mathfrak g_0$-module. It is well-known that $(\mathfrak g_\epsilon)_0=\mathfrak h_\epsilon$ and, if $\mu\neq0,$ then $(\mathfrak g_\epsilon)_\mu$ is one-dimensional for all $\epsilon.$

\subsection{A Chevalley basis}

We construct a basis for $(\mathfrak g_\epsilon)_\mu$, where $\mu\in\textnormal{wt}(\mathfrak g_\epsilon),$ in terms of the fixed Chevalley basis for $\mathfrak g$. 

It is well known that if $\mu\neq0$, then $\mu=\alpha|_{\mathfrak h_0}$ for some $\alpha\in R$ and $\alpha|_{\mathfrak h_0}=\beta|_{\mathfrak h_0}$ if and only if $\beta=\sigma^j(\alpha)$ for some $j\in\{0,1,\ldots,k-1\}.$

First suppose that $\mathfrak g$ is not of type $A_{2n}$. Then, given $\alpha\in R^+$ and $\epsilon\in\{0,\ldots,k_\alpha-1\}$, define
$$x_{\alpha,\epsilon}^\pm:= \displaystyle\sum_{j=0}^{k_\alpha-1}\xi^{{-}j\epsilon}x_{\sigma^j(\alpha)}^\pm   \text{ and } \ H_{\alpha,\epsilon}:= \displaystyle\sum_{j=0}^{k_\alpha-1}\xi^{{-}j\epsilon}h_{\sigma^j(\alpha)}$$

Then
\begin{equation}\label{grels}
    x_{\sigma(\alpha),\epsilon}^\pm=\xi^{-\epsilon}x_{\alpha,\epsilon}^\pm,\ \ H_{\sigma(\alpha),\epsilon}=\xi^{-\epsilon}H_{\alpha,\epsilon},\ \ x_\alpha^\pm=\frac{1}{k_\alpha}\sum_{\epsilon=0}^{k-1}x_{\alpha,\epsilon}^\pm,\ \text{ and }\  h_\alpha=\frac{1}{k_\alpha}\sum_{\epsilon=0}^{k-1}H_{\alpha,\epsilon.}
\end{equation}

Now assume that $\mathfrak g$ is of type $A_{2n}$. Given $\alpha\in R^+$ and $\epsilon \in \{0,1\}$, define
$$x_{\alpha,\epsilon}^\pm:=\begin{cases}\delta_{1,\epsilon}x_\alpha^\pm, &\text{ if }\alpha =\sigma(\alpha) \\   x_\alpha^\pm+(-1)^\epsilon x_{\sigma(\alpha)}^\pm,&\text{ if }\alpha|_{\mathfrak h_0}\in R_l,\\\sqrt{2}\left(x_\alpha^\pm+(-1)^\epsilon x_{\sigma(\alpha)}^\pm\right),&\text{ if }\alpha|_{\mathfrak h_0}\in R_s,\end{cases}$$
and
$$H_{\alpha,\epsilon}:=\begin{cases}\delta_{0,\epsilon}h_\alpha, &\text{ if }\alpha =\sigma(\alpha) \\  h_\alpha+(-1)^\epsilon h_{\sigma(\alpha)},&\text{ if }\alpha|_{\mathfrak h_0}\in R_l,\\2\left(h_\alpha+(-1)^\epsilon h_{\sigma(\alpha)}\right),&\text{ if }\alpha|_{\mathfrak h_0}\in R_s.\end{cases}$$
Moreover, if $\mathfrak g$ is of type $A_{2n}$ and $\beta+\sigma(\beta)\in R^+$, then $\beta|_{\mathfrak h_0} \in R_{s}$ and $(\beta + \sigma(\beta))|_{\mathfrak h_0} \in 2R_{l}$.  In this case, we have
\begin{equation}\label{s}
x_{\beta+\sigma(\beta),1}^\pm
= \frac{s}{4}[x_{\beta,0}^\pm,x_{\beta,1}^\pm]
= -s[ x_\beta^\pm,x_{\sigma(\beta)}^\pm]
\quad \textup{ for some $s \in \{ -1, 1\}$}.
\end{equation}
Note that similar relations to those in \eqref{grels} also hold in this case.

If $\mu=\alpha|_{\mathfrak h_0}$ is a weight of $\mathfrak g_\epsilon$ then the corresponding weight space $\left(\mathfrak g_\epsilon\right)_{\pm\mu}$ has basis $\left\{x_{\alpha,\epsilon}^\pm\right\}$. Also, $\mathfrak h_\epsilon$ is spanned by  $\left\{H_{\alpha_i,\epsilon}\mid i\in I\right\}$. Choose $O$ as a complete set of representatives of the orbits of the $\sigma$-action on $R^+$. Let $\mu\in\text{wt}\left(\mathfrak g_\epsilon\right)\cap Q_0^+\setminus \{0\}$, then define $x_{\mu,\epsilon}^\pm:=x_{\alpha,\epsilon}^\pm$ where $\alpha\in O$ is such that $\alpha|_{\mathfrak h_0}=\mu$. For convenience, define $x_\mu^\pm:=0$ if $\mu\notin\text{wt}\left(\mathfrak g_\epsilon\right)\cap Q_0^+\setminus \{0\}$. Also, there is an injective map $o:I_0\to I$ such that $\alpha_{o(i)}\in O$ for all $i\in I_0$. Given $i\in I_0$ and $\epsilon\in\{1,\ldots,k-1\}$ define $$h_{i,\epsilon}=\left\{\begin{array}{cc}\frac{1}{2}H_{\alpha_{o(i)},\epsilon},&\text{ if }\mathfrak g\text{ is of type }A_{2n}\text{ and }\alpha_i\in R_s,\\H_{\alpha_{o(i)},\epsilon},&\text{ otherwise}.\end{array}\right.$$

Define
$$\mathcal C^\sigma(O):=\left\{x_{\mu,\epsilon}^\pm,h_{i,\epsilon}\mid\epsilon\in\{0,\ldots,k-1\},\mu\in\text{wt}(\mathfrak g_\epsilon)\cap Q_0^+\setminus \{0\},i\in I_0\right\}\setminus \{0\}.$$
Then $\mathcal C^\sigma(O)$ is a basis for $\mathfrak g$. If $\mathfrak g$ is of type $A_{2n}$, we choose $O$ such that $s=1$ in \eqref{s}. If $\mathfrak g$ is of type $D_4$ and $k=3$ we need to choose $O$ in a more specific way in order to ensure that the brackets of elements of $\mathcal C^\sigma(O)$ are in the $\mathbb Z$-span of elements of $\mathcal C^\sigma(O)$. To that end, let $j\in I$ be the unique vertex fixed by $\sigma$ and choose $i\in I$ such that $\sigma(i)\neq i.$ Define 
$$O_i:=\left\{\alpha\in R^+\mid \sigma(\alpha)=\alpha\right\}\cup\left\{\alpha_i,\alpha_j+\alpha_i,\alpha_j+\sigma(\alpha_i)+\sigma^2(\alpha_i)\right\}.$$
To ease the notational burden, we define $x_{\mu,\epsilon}^\pm:=0$ if $\mu$ is not a weight of $\mathfrak g_\epsilon$. Also, if $\alpha\in O$ is such that $\mu=\alpha|_{\mathfrak h_0}$, define   $$h_{\mu,\epsilon}:=\left\{\begin{array}{cc}\frac{1}{2}H_{\alpha,\epsilon},&\text{ if }\mathfrak g\text{ is of type }A_{2n}\text{ and }\alpha\in R_s,\\H_{\alpha,\epsilon},&\text{ otherwise}.\end{array}\right..$$

The following lemma encodes the bracket information for elements of $\mathcal C^\sigma(O)$.
\begin{lem}\label{brackets}
    Let $\epsilon,\epsilon'\in\{0,\ldots,k-1\}$, $\mu\in R_0^+$, $\nu\in\textnormal{wt}(\mathfrak g_\epsilon)\cap Q_0^+\setminus \{0\}$, and $\eta\in\textnormal{wt}(\mathfrak g_\epsilon)\cap\textnormal{wt}(\mathfrak g_{\epsilon'})\cap Q_0^+\setminus \{0\}$. Then
    \begin{enumerate}
        \item $\left[h_{\mu,0},x_{\nu,\epsilon}^\pm\right]=\pm\nu\left(h_{\mu,0}\right)x_{\nu,\epsilon}^\pm$;
        
        \item $\displaystyle\left[x_{\eta,\epsilon}^+,x_{\eta,\epsilon'}^-\right]=\left\{\begin{array}{cc}2h_{\eta,\overline{\epsilon+\epsilon'}},&\text{if }\eta\in R_s\text{ and }\mathfrak g\text{ is of type }A_{2n},\\\delta_{\epsilon\epsilon',1}h_{\eta/2,0},&\text{if }\eta\in 2R_s\text{ and }\mathfrak g\text{ is of type }A_{2n},\\ h_{\eta,\overline{\epsilon+\epsilon'}},&\text{ otherwise};\end{array}\right.$
        
        \item If $h_{\nu,1}\neq0$, then $\displaystyle\left[h_{\nu,1},x_{\nu,\epsilon}^\pm\right]=\left\{\begin{array}{cc}\pm3x_{\nu,\overline{\epsilon+1}}^\pm,&\text{ if }\nu\in R_s\text{ and }\mathfrak g\text{ is of type }A_{2n},\\\pm2x_{\nu,\overline{\epsilon+1}}^\pm,&\text{ otherwise}.\end{array}\right.$
        
        \item $\left[(\mathfrak g_\epsilon)_\mu,(\mathfrak g_{\epsilon'})_\eta\right]\subseteq(\mathfrak g_{{\epsilon+\epsilon'}})_{\mu+\eta}$.
    \end{enumerate}
\end{lem}

\begin{rem}
    If $\mathfrak g$ is not of type $A_{2n}$, then $\mathcal C_0^\sigma(O):=\left\{x_{\mu,0}^\pm,h_{i,0}\mid\mu\in R_0^+,i\in I_0\right\}$ is a Chevalley basis for $\mathfrak g_0$. If $\mathfrak g$ is of type $A_{2n}$, then we replace $h_{j,0}\in\mathcal C_0^\sigma(O)$ where $j\in I_0$ is such that $\alpha_j\in R_s$, by $2h_{j,0}$. 
\end{rem}

\subsection{Twisted multiloop algebras}

Fix $m\in\mathbb N$ and define $\mathbb C\langle m\rangle=\mathbb C\left[t_1^{\pm 1},\ldots,t_m^{\pm 1}\right]$ to be the multivariable Laurent polynomials. Given $\mathbf r\in\mathbb Z^m$ define $t^{\mathbf r}:=t_1^{r_1}\dots t_m^{r_m}\in\mathbb C\langle m\rangle$. The (untwisted) multiloop algebra of $\mathfrak{g}$ is $\mathfrak{g}\otimes\mathbb C\langle m\rangle$ with bracket given by linearly extending the bracket
$$\left[x\otimes f,y\otimes g\right]=[x,y]_{\mathfrak g}\otimes fg,\ \forall x,y\in\mathfrak g,\ f,g,\in\mathbb C\langle m\rangle.$$
Given any subalgebra $\mathfrak a\subset\mathfrak g$, define $\mathfrak a\langle m\rangle:=\mathfrak a\otimes\mathbb C\langle m\rangle$. Then 
$$\mathfrak g\langle m\rangle=\mathfrak n^-\langle m\rangle\oplus\mathfrak h\langle m\rangle\oplus\mathfrak n^+\langle m\rangle.$$

Let $\xi$ be a primitive $k$th root of unity. Then $\sigma$ also acts on $\mathbb C\langle m\rangle$ via $\sigma\left(f\left(t_1,\ldots,t_m\right)\right)=f\left(\xi^{-1}t_1,\ldots,t_m\right)$. 
These two group actions give rise to an action of $\sigma$ on the (untwisted) multiloop algebra of $\mathfrak{g}$, $\mathfrak{g}\langle m\rangle$, via $\sigma(x\otimes f)=\sigma(x)\otimes\sigma(f)$, for all $x\in\mathfrak{g}$ and $f\in\mathbb C\langle m\rangle$. The \textbf{twisted multiloop algebra of $\mathfrak{g}$}, $\mathcal T_m^\sigma(\mathfrak{g})$, is the subalgebra of $\mathfrak g\langle m\rangle$, which is fixed by this action of $\sigma$. 

\begin{rem}
    Note that, in the case $\sigma=1$, $\mathcal T_m^\sigma(\mathfrak{g})=\mathfrak g\langle m\rangle$ and the integral bases in this case have been done in \cite{C}. Given this, we will assume going forward that $k\in\{2,3\}$ i.e. $\sigma$ is nontrivial.
\end{rem}

Given any Lie subalgebra $\mathfrak a\subset\mathfrak g$, define $\mathcal T_{m,\epsilon}^\sigma(\mathfrak{a})\subset\mathcal T_{m}^\sigma(\mathfrak{g})$ by 
$$\mathcal T_{m,\epsilon}^\sigma(\mathfrak{a}):=\mathfrak a_\epsilon\otimes t_1^{\epsilon}\mathbb C\left[t_1^{\pm k},t_2^{\pm 1},\ldots,t_m^{\pm 1}\right].$$
Note that 
$$\mathcal T_m^\sigma(\mathfrak{g})=\bigoplus_{\epsilon=0}^{k-1}\mathcal T_{m,\epsilon}^\sigma(\mathfrak{g})\ \text{ and }\ \mathcal T_{m,\epsilon}^\sigma(\mathfrak{g})=\mathcal T_{m,\epsilon}^\sigma(\mathfrak{n}^-)\oplus\mathcal T_{m,\epsilon}^\sigma(\mathfrak{h})\oplus\mathcal T_{m,\epsilon}^\sigma(\mathfrak{n}^+).$$

Given $\mathbf r\in\mathbb Z^m$ and $\mu\in\textnormal{wt}(\mathfrak g_{r_1})\cap Q_0^+\setminus \{0\}$ define 
\begin{equation}\label{newelts}
   x_{\mu,\mathbf r}^{\pm}:=x_{\mu,r_1}^\pm\otimes t^\mathbf r
    \text{ and }\ h_{\mu,\mathbf r}:=h_{\mu,r_1}\otimes t^\mathbf r
\end{equation}
and for $\mu_i = \alpha_i|_{\mathfrak h_0}, i \in I_0$, 
\begin{equation}
   x_{i,\mathbf r}^{\pm}:=x_{\mu_i,r_1}^\pm\otimes t^\mathbf r
    \text{ and }\ h_{i,\mathbf r}:=h_{\mu_i,r_1}\otimes t^\mathbf r
\end{equation}
The set $$B:=\left\{x_{\mu,\mathbf r}^\pm,h_{i,\mathbf r}\mid\mathbf r\in\mathbb Z^m,\mu\in\text{wt}(\mathfrak g_{r_1})\cap Q_0^+\setminus \{0\},i\in I_0\right\}\setminus \{0\}$$
is a basis for $\mathcal T_m^\sigma(\mathfrak g)$.

\subsection{Universal enveloping algebras}

Given a Lie algebra $\mathfrak a$ let $U(\mathfrak a)$ denote its universal enveloping algebra. Given $u\in U(\mathfrak a)$ and $r\in\mathbb Z_+$ define the divided powers and binomial coefficients as follows
$$u^{(r)}=\frac{u^r}{r!}\text{ and }\binom{u}{r}:=\frac{u(u-1)\dots(u-r+1)}{r!}.$$
Define $T^0(\mathfrak a):=\mathbb C$, and for all $j\in\mathbb N$, define $T^j(\mathfrak a):=\mathfrak a^{\otimes j}$, $T(\mathfrak a):=\bigoplus_{j=0}^\infty T^j(\mathfrak a)$, and $T_j(\mathfrak a):=\bigoplus_{k=0}^jT^k(\mathfrak a)$. Then, set $U_j(\mathfrak a)$ to be the image of $T_j(\mathfrak a)$ under the canonical surjection $T(\mathfrak a)\to U(\mathfrak a)$, and for any $u\in U(\mathfrak a)$ \emph{define the degree of $u$} by $$\deg u:=\min_{j}\{u\in U_j(\mathfrak a)\}.$$

The Poincar\'e-Birkhoff-Witt (PBW) Theorem gives the following isomorphisms
\begin{eqnarray*}
    U(\mathfrak g)&\cong&U(\mathfrak n^-)\otimes U(\mathfrak h)\otimes U(\mathfrak n^+)\\
    U(\mathfrak g\langle m\rangle)&\cong&U(\mathfrak n^-\langle m\rangle)\otimes U(\mathfrak h\langle m\rangle)\otimes U(\mathfrak n^+\langle m\rangle).
\end{eqnarray*}

Let $U_{\mathbb Z}\left(\mathcal T^\sigma_m(\mathfrak g)\right)$ be the $\mathbb Z$-subalgebra generated by $$\left\{\left(x_{\mu,\mathbf r}^\pm\right)^{(r)}\mid\mathbf r\in\mathbb Z^m,\ \mu\in\textnormal{wt}(\mathfrak g_{r_1})\cap Q_0^+\setminus \{0\},\ r\in\mathbb Z_+\right\}.$$

This set of generators for the hyperalgebras can be conveniently reduced:
 
\begin{prop} The $\mathbb Z$-subalgebra $U_{\mathbb Z}\left(\mathcal T^\sigma_m(\mathfrak g)\right)$ is generated by
$$\left\{\left(x_{i,\mathbf r}^\pm\right)^{(r)}\mid\mathbf r\in\mathbb Z^m, i \in I_0,\ r\in\mathbb Z_+\right\}.$$

\end{prop}
\begin{proof}
    The proof is analogous to \cite[Corollary 4.4.12]{M} by using induction on the height of the associated roots.
\end{proof}

\section{Main Results} \label{main}

Given $\mathbf r\in\mathbb Z^m$, $\mu\in\textnormal{wt}(\mathfrak g_{r_1})\cap Q_0^+\setminus \{0\}$, and $l\in\mathbb Z_+$, define $\Lambda_{\mu,\mathbf r,l}^\sigma\in U(\mathbb T_m^\sigma(\mathfrak g))$ to be the coefficient of $u^l$ in the series 
$$\Lambda_{\mu,\mathbf r}^\sigma(u):=\exp\left(-\sum_{j\geq1}\frac{h_{\mu,j\mathbf r}}{j}u^j\right).$$

Again, for $\mu_i = \alpha_i|_{\mathfrak h_0}, i \in I_0$, we set 
$$\Lambda_{i,\mathbf r}^\sigma := \Lambda_{\mu_i,\mathbf r}^\sigma \text{ and } \Lambda_{i,\mathbf r,l}^\sigma:=\Lambda_{\mu_i,\mathbf r,l}^\sigma .$$

\

The following Proposition is clear.

\begin{prop}\label{prop:repeat}
    For all $\mathbf r,\mathbf s\in\mathbb C\langle m\rangle$, $\mu\in\textnormal{wt}(\mathfrak g_{r_1})\cap Q_0^+\setminus \{0\}$, $\nu\in\textnormal{wt}(\mathfrak g_{s_1})\cap Q_0^+\setminus \{0\}$, and $r,s\in\mathbb Z_+$
    \begin{eqnarray*}
        \Lambda_{\mu,\mathbf r,r}^\sigma\Lambda_{\nu,\mathbf s,s}^\sigma
        &=&\Lambda_{\nu,\mathbf s,s}^\sigma\Lambda_{\mu,\mathbf r,r}^\sigma\\
        \left(x_{\mu,\mathbf r}^\pm\right)^{(r)}\left(x_{\mu,\mathbf r}^\pm\right)^{(s)}
        &=&\binom{r+s}{s}\left(x_{\mu,\mathbf r}^\pm\right)^{(r+s)}
    \end{eqnarray*}
    \hfill\qedsymbol
\end{prop}

 The following lemma shows that elements not all elements of the form $\Lambda_{\alpha,\mathbf r,l}$ are $\mathbb Z$-linearly independent: 
 
\begin{lem}\label{f^kintermsoff}
Let $i\in I_0$, $\mathbf r \in \mathbb Z^m$, and $l\in\mathbb Z_+$. Then
\begin{equation*}
	\Lambda_{i,\mathbf r,l}^\sigma \in \mathbb Z[\Lambda_{i,\mathbf s,s }^\sigma \mid \gcd(\mathbf s)=1 \text{ and } s\in\mathbb Z_+ \} 
\end{equation*}

\end{lem}
\begin{proof}
The lemma is proven in $\mathcal T_1^\sigma(\mathfrak{g})$ for $f=t_1$ as part of Section 4.4 in \cite{M}. We can extend it to the current setting by replacing $t_1$ with $t^\mathbf r$.
\end{proof}

In order to use Mitzman's results we shall need few specific morphisms of $U_{\mathbb Z}\left(\mathcal T^\sigma_m(\mathfrak g)\right)$:

\begin{itemize}
    \item[i)] for each $\mathbf v\in\mathbb Z^m$, let $T_\mathbf v: U_{\mathbb Z}\left(\mathcal T^\sigma_m(\mathfrak g)\right) \to U_{\mathbb Z}\left(\mathcal T^\sigma_m(\mathfrak g)\right)$ be the automorphism defined by $$T_\mathbf v\left(x_{i,\mathbf r}^\pm\right):=x_{i,\mathbf r\mp\mathbf v}^\pm,$$ 
    for all $i\in I_0$ and $\mathbf r\in\mathbb Z^m$. Then $T_\mathbf v\left(h_{i,\mathbf r}\right)=h_{i,\mathbf r}$.

    \item[ii)] let $\Omega: U_{\mathbb Z}\left(\mathcal T^\sigma_m(\mathfrak g)\right) \to U_{\mathbb Z}\left(\mathcal T^\sigma_m(\mathfrak g)\right)$ be the antiautomorphism defined by $$\Omega\left(x_{\mu,\mathbf r}^\pm\right)=x_{\mu,\mathbf r}^\mp,$$ for all $\mathbf r\in\mathbb Z^m$ and $\mu\in\textnormal{wt}(\mathfrak g_{r_1})\cap Q_0^+\setminus \{0\}$. Then $\Omega\left(h_{\mu,\mathbf r}\right)=h_{\mu,\mathbf r}$.

    \item[iii)]  for each $\mathbf b\in\mathbb Z^m$, let $\lambda_\mathbf b: U_{\mathbb Z}\left(\mathcal T^\sigma_m(\mathfrak g)\right) \to U_{\mathbb Z}\left(\mathcal T^\sigma_m(\mathfrak g)\right)$ be the injective homomorphism defined by $$\lambda_\mathbf b\left(x_{\mu,\mathbf r}^\pm\right)=x_{\mu,\mathbf b\mathbf r}^\pm,$$ for all $\mathbf r\in\mathbb Z^m$ and $\mu\in\textnormal{wt}(\mathfrak g_{r_1})\cap Q_0^+\setminus \{0\}$, where the multiplication of vectors is componentwise multiplication. Then $\lambda_\mathbf b\left(h_{\alpha,\mathbf r}\right)=h_{\alpha,\mathbf b\mathbf r}$. This is really just a variable replacement replacing $t_i$ with $t_i^{b_i}$. Note that if $b_1$ is odd, then $\lambda_\mathbf b$ is an automorphism.
\end{itemize}

Now, set  
$$M_m^\sigma (\mathfrak g):=\left\{\left(x_{\mu,\mathbf r}^\pm\right)^{(r)},\ \Lambda_{i,\mathbf s,s}^\sigma,\ \tbinom{h_{i,\mathbf 0}}{r}\mid \mathbf r,\mathbf s\in\mathbb Z^m, \gcd(\mathbf s)=1, \ \mu\in\textnormal{wt}(\mathfrak g_{r_1})\cap Q_0^+\setminus \{0\},\ i\in I_0,\ r,s\in\mathbb Z_+\right\}$$
and define a monomial in $U_{\mathbb Z}(\mathbb T_m^\sigma(\mathfrak g))$ to be any finite product of elements of the set $M_m^\sigma (\mathfrak g)$.

\begin{prop}\label{prop:arrange}
    If $r,s\in\mathbb Z_+$, $i,j\in I_0$, and $\mathbf r,\mathbf s\in\mathbb Z^m$, then
    \begin{eqnarray}
        \left(x_{\alpha_j,\mathbf r}^+\right)^{(r)}\Lambda_{\alpha_i,\mathbf s,s}^\sigma
        &=&\Lambda_{\alpha_i,\mathbf s,s}^\sigma\left(x_{\alpha_j,\mathbf r}^+\right)^{(r)}+P\label{x+Lambda}\\
        \Lambda_{\alpha_i,\mathbf s,s}^\sigma\left(x_{\alpha_j,\mathbf r}^-\right)^{(r)}
        &=&\left(x_{\alpha_j,\mathbf r}^-\right)^{(r)}\Lambda_{\alpha_i,\mathbf s,s}^\sigma+Q\label{Lambdax-}\\
        \left(x_{\alpha_i,\mathbf r}^\pm\right)^{(r)}\left(x_{\alpha_j,\mathbf s}^\pm\right)^{(s)}
        &=&\left(x_{\alpha_j,\mathbf s}^\pm\right)^{(s)}\left(x_{\alpha_i,\mathbf r}^\pm\right)^{(r)}+R\label{x+x+}\\
        \left(x_{\alpha_i,\mathbf r}^+\right)^{(r)}\left(x_{\alpha_i,\mathbf s}^-\right)^{(s)}
        &=&\left(x_{\alpha_i,\mathbf s}^-\right)^{(s)}\left(x_{\alpha_i,\mathbf r}^+\right)^{(r)}+S\label{x+x-}
    \end{eqnarray}
    where $P$, $Q$, $R$, and $S$ are integer linear combinations of monomials with degree less than $r+s$.
\end{prop}
\begin{proof}
    By a result in Mitzman \cite[Lemma 4.4.1 (viii)]{M}, $$\left(x_{\alpha_j,0}^+\otimes 1\right)^{(r)}\Lambda_{\alpha_i,t,s}^\sigma=\Lambda_{\alpha_i,t,s}^\sigma\left(x_{\alpha_j,0}^+\otimes 1\right)^{(r)}+P'$$
    where $P'$ is a integer linear combination of monomials of degree less than $r+s$. Replacing $t$ with $t_1\cdots t_m$ and using $\left(x_{\alpha_j,\mathbf r}^+\right)^{(r)}\Lambda_{\alpha_i,\mathbf s,s}^\sigma=T_{-\mathbf r}\lambda_{\mathbf s}\left(\left(x_{\alpha_j,0}^+\otimes 1\right)^{(r)}\Lambda_{\alpha_i,\mathbf 1,s}^\sigma\right)$, where $\mathbf 1:=(1,\ldots,1)\in\mathbb Z^m$ gives \eqref{x+Lambda}. Then $\Omega\left(\left(x_{\alpha_j,\mathbf r}^+\right)^{(r)}\Lambda_{\alpha_i,\mathbf s,s}^\sigma\right)=\Lambda_{\alpha_i,\mathbf s,s}^\sigma\left(x_{\alpha_j,\mathbf r}^-\right)^{(r)}$ gives \eqref{Lambdax-}. Again by another result in \cite[Lemma 4.4.1 (v)]{M},
    $$\left(x_{\alpha_i}^\pm\otimes1\right)^{(r)}\left(x_{\alpha_j}^\pm\otimes t\right)^{(s)}=\left(x_{\alpha_j}^\pm\otimes t\right)^{(s)}\left(x_{\alpha_i}^\pm\otimes1\right)^{(r)}+R'$$
    where $R'$ is an integral linear combination of monomials of degree less than $r+s$. Replacing $t$ with $t_1\cdots t_m$ and using    $\left(x_{\alpha_i,\mathbf r}^\pm\right)^{(r)}\left(x_{\alpha_j,\mathbf s}^\pm\right)^{(s)}=T_{\mp\mathbf r}\lambda_{-\mathbf r+\mathbf s}\left(\left(x_{\alpha_i,0}^\pm\right)^{(r)}\left(x_{\alpha_j,\mathbf 1}^\pm\right)^{(s)}\right)$ gives \eqref{x+x+}. Finally, by \cite[Lemma 4.4.1 (iv)]{M}, $$\left(x_{\alpha_i}^+\otimes 1\right)^{(r)}\left(x_{\alpha_i}^-\otimes t\right)^{(s)}=\left(x_{\alpha_i}^-\otimes t\right)^{(s)}\left(x_{\alpha_i}^+\otimes 1\right)^{(r)}+S'$$
    where $S'$ is an integer linear combination of monomials of degree less than $r+s$. Replacing $t$ with $t_1\cdots t_m$ and using $\left(x_{\alpha_i,\mathbf r}^+\right)^{(r)}\left(x_{\alpha_i,\mathbf s}^-\right)^{(s)}=T_{-\mathbf r}\lambda_{\mathbf r+\mathbf s}\left(\left(x_{\alpha_i,0}^+\right)^{(r)}\left(x_{\alpha_i,\mathbf 1}^-\right)^{(s)}\right)$ gives \eqref{x+x-}
\end{proof}

Given an order on the basis $B$ of $\mathcal T_m^\sigma(\mathfrak{g})$ and a PBW monomial with respect to this order, we construct an ordered monomial in the elements of the set
$M_m^\sigma(\mathfrak g)$,
using the correspondence $\left(x_{\mu,\mathbf r}^\pm\right)^{r} \leftrightarrow \left(x_{\mu,\mathbf r}^\pm\right)^{(r)}$, $h_{i,\mathbf 0}^r \leftrightarrow \binom{h_{i,\mathbf 0}}{r}$ and $(h_{i,\mathbf r} )^k \leftrightarrow \Lambda_{i,\mathbf r,k}^\sigma$ for $\mathbf r\ne 0$. The set of ordered monomials thus obtained is a basis of $U(\mathcal T_m^\sigma(\mathfrak{g}))$.

\

The next theorem extends Mitzman's Theorem  \cite{M} (see also \cite{Pre}) to the twisted multiloop algebras.  

\begin{thm} \label{integralforms}
The subalgebra $U_{\mathbb Z}(\mathcal T_m^\sigma(\mathfrak{g}))$ is a free $\mathbb Z$-module and the sets of ordered monomials constructed from $M_m^\sigma(\mathfrak g)$ is a $\mathbb Z$-basis of $U_{\mathbb Z}(\mathcal T_m^\sigma(\mathfrak{g}))$.\end{thm}

\begin{proof}
    Let $\mathcal B$ be the set of ordered monomials constructed from $M_m^\sigma (\mathfrak g)$. By PBW--Theorem, $\mathcal B$ is a $\mathbb C$-linearly independent set. Therefore, it is a ${\mathbb Z}$-linearly independent set. 
    
    We proceed by induction on the degree of monomials in $U_{\mathbb Z}(\mathcal T_m^\sigma(\mathfrak{g}))$ to show that the ${\mathbb Z}$-span of $\mathcal B$ is $U_{\mathbb Z}(\mathcal T_m^\sigma(\mathfrak{g}))$: first, by definition, any degree one monomial is in the ${\mathbb Z}$-span of $\mathcal B$. Now, let $m$ be any monomial in $U_{\mathbb Z}(\mathcal T_m^\sigma(\mathfrak{g}))$. If $m\in\mathcal B$, then we are done. If not, then either the factors of $m$ are not in the correct order or $m$ has certain factors as those appearing on the left-hand side of the statements of Proposition \ref{prop:repeat}.
    If the factors of $m$ are not in the correct order, then we rearrange them using Proposition \ref{prop:arrange}, producing ${\mathbb Z}$-linear combinations of monomials in the correct order with (possibly) other monomials of lower degree.
    All these lower degree monomials are in the ${\mathbb Z}$-span of $\mathcal B$ by the induction hypothesis. So, we conclude that $m\in{\mathbb Z}$-span $\mathcal B$. 
    
    Therefore, the ${\mathbb Z}$-span of $\mathcal B$ is $U_{\mathbb Z}(\mathcal T_m^\sigma(\mathfrak{g}))$ and, hence, $\mathcal B$ is an integral basis for $U_{\mathbb Z}(\mathcal T_m^\sigma(\mathfrak{g}))$.
\end{proof}

\end{document}